\documentclass[12pt]{amsart}

\usepackage[margin=1in]{geometry}  
\usepackage{graphicx}              

\usepackage{amssymb}
\usepackage{esint}

\newtheorem{thm}{Theorem}[section]
\newtheorem{prop}[thm]{Proposition}
\newtheorem{lem}[thm]{Lemma}
\newtheorem{cor}[thm]{Corollary}

\theoremstyle{definition}
\newtheorem{defi}[thm]{Definition}

\theoremstyle{remark}
\newtheorem{rem}[thm]{Remark}

\theoremstyle{example}
\newtheorem{ex}[thm]{Example}

\numberwithin{equation}{section}

\newcommand{\TDer}{\nabla_\tau} 
\newcommand{\RDer}{\frac{\partial}{\partial r}} 

\newcommand{\FBA}{\partial\{u_1>u_2\}} 

\newcommand{\Tprj}{T_{p,r}u_j} 
\newcommand{\Torj}{T_{0,r}u_j} 

\newcommand{\Fpr}{W(\{u_j\},p,r)} 
\newcommand{\EFpr}{\frac{1}{r^{d+2}}\int_{B_r(p)}\sum(|\nabla u_j|^2+2 f_ju_j)-\frac{2}{r^{d+3}}\int_{\partial B_r(p)}\sum u_j^2} 
\newcommand{\For}{W(\{u_j\},0,r)}

\begin{document}

\title{On the Multiple Membranes Problem}

\author{Ovidiu Savin}
\address{Department of Mathematics, Columbia University}
\email{savin@math.columbia.edu}
\author{Hui Yu}
\address{Department of Mathematics, Columbia University}
\email{huiyu@math.columbia.edu}

\begin{abstract}
We establish the optimal regularity for solutions to the multiple membranes problem, and perform a blow-up analysis at points on the free boundary with the highest multiplicity. This leads to a complete classification of blow-up profiles in the plane. The main technical tool is a Weiss-type monotonicity formula. 
\end{abstract}

 \maketitle

\tableofcontents

\section{Introduction}
For a positive integer $N$, the $N$-membranes problem is the study of the equilibrium position of $N$ elastic membranes subject to forcing and boundary conditions.  If heights of the membranes are described  by real functions $u_1,u_2,\dots,u_N$ over a domain $\Omega$ inside $\mathbb{R}^d$ with the force acting on $u_j$ given by $f_j$, then in the most basic model $u_j$'s minimise the following energy \begin{equation*}
\mathcal{F}(\{u_j\})=\int_{\Omega}\sum( |\nabla u_j|^2+2 f_ju_j)dx
\end{equation*}  among the convex class of functions satisfying $u_1\ge u_2\ge \cdots\ge u_N$ as well as  boundary conditions.  Here the ordering is a consequence of the non-penetrating property of the membranes. The problem becomes particularly interesting if $f_1>f_2>\cdots>f_N$ since then the membranes tend to contact each other in certain regions, giving rise to free boundaries. 

Under reasonable assumptions on $f_j$'s and compatible boundary conditions, it is not too difficult to produce minimisers by the direct method, say, in the space of Sobolev functions. Uniqueness also follows rather easily from convexity of the functional \cite{CV}. The interesting problems are then the regularity of solutions $\{u_j\}$ and of the free boundaries $\partial\{u_j>u_{j+1}\}$, which lie at the transition region between where the two membranes detach and where they are in contact. 

These two problems are well-understood now for the case when $N=2$ thanks to the work of Silvestre \cite{Sil}. Under the assumptions that $f_j$'s are H\"older continuous and satisfy the non-degeneracy condition $f_1>f_2$,  he established interior $C^{1,1}$-regularity of solutions, and showed that the free boundary $\FBA$ is locally a $C^{1,\alpha}$  hypersurface outside possible cusps. The main observation is that under such conditions the difference $u_1-u_2$ solves the obstacle problem, and techniques from the classical work of Caffarelli \cite{Caf} can be applied. Note that $C^{1,1}$-regularity is optimal since the Hessians of the solutions are in general not continuous across the free boundary. This is due to the physical fact that  in the contact set the two forces are evenly {\em shared} by the two membranes, while in the non-contact region each membrane sees only one force term.  For the case of two membranes, one even has a complete picture when the operators acting on the two functions are of different orders as in the work of Caffarelli, De Silva and the first author \cite{CDS}.

Considering the success with the case $N=2$,  the problem with more than two membranes is surprisingly challenging. The major difficulty, comparing with other obstacle-type problems, is the jump in the right-hand side {\em within the same phase}. To illustrate the idea, consider the model case where $N=3$, $f_1=1$, $f_2=0$ and $f_3=-1$. Within the region $\{u_1>u_2\}$, the difference $u_1-u_2$ solves no nice equation since $\Delta u_2$ is either $0$ or $-1/2$ depending on whether it is in contact with $u_3$ or not. Therefore, the jump in the right-hand side of the equation for $u_1-u_2$ does not correspond to a jump across the free boundary involving $u_1$ and $u_2$.  This  `phantom jump' of the right hand side makes it difficult to apply standard techniques for obstacle-type free boundary problems. 

As a result,  although the past three decades has seen quite some works \cite{ARS}\cite{CCV}\cite{LR}, there has been no essential progress on regularity of solutions since the original work of Chipot and Vergara-Caffarelli \cite{CV}, where they proved, for bounded forcing terms, interior $W^{2,p}$-regularity of solutions for any $p<\infty$, which leads to the sub-optimal  $C^{1,\alpha}$-regularity for any $\alpha\in(0,1)$. Almost nothing is known concerning the regularity of the free boundaries other than that they have null Lebesgue measure as in the work of Lindgren and Razani \cite{LR}. 

In this work, we study optimal regularity of solutions and classify blow-up profiles at the free boundary points for $N\ge 3$.  

For H\"older continuous forcing terms $f_j$, we first establish the optimal $C^{1,1}$ interior regularity of the solution $\{u_j\}_{j=1,2,\dots, N}$. The key idea is that at a contact point $p$ of the highest multiplicity, that is,  $p\in \partial\{u_j>u_{j+1}\}$ for all $j$, the known $C^{1,\alpha}$-regularity gives $u_j(p)=u_{j+1}(p)$ as well as $\nabla u_j(p)=\nabla u_{j+1}(p)$. This leads to a comparison between the Hessians for every consecutive pairs of functions $$D^2u_j(p)\ge D^2u_{j+1}(p).$$ On the other hand, the growth of $u_j-u_{j+1}$ from $p$ is at most quadratic as shown by Lindgren and Razani \cite{LR}, that is,  $$D^2u_j(p)\le D^2u_{j+1}(p)+C$$ for some constant $C$.  Combined with a $C^{2,\alpha}$-estimate for the sum $\sum u_j$, these are sufficient to bound the Hessians at  contact points of the highest multiplicity. Outside these points, the problem reduces to a configuration involving fewer membranes, allowing an induction argument, which leads to 
\begin{thm}
Suppose $f_j\in C^\alpha(\Omega)$ for some $\alpha\in(0,1)$ and for each $j=1,2,\dots, N$, then for any compact subset $K$ of $\Omega$, the solution to the $N$-membranes problem, $\{u_j\}_{j=1,2,\dots, N},$ satisfies \begin{equation*}
\sum_{j=1}^N\|u_j\|_{C^{1,1}(K)}\le C\sum_{j=1}^N(\|u_j\|_{\mathcal{L}^{\infty}(\Omega)}+\|f_j\|_{C^{\alpha}(\Omega)}),
\end{equation*} where the constant $C$ depends on $K$, $d, N$ and $\alpha$. 
\end{thm} In the rest of the paper, any constants depending only on $d$ and $N$ will be called {\em universal}. 
 
This opens the door to the study of free boundary regularity. As a first step, we need to classify blow-up profiles at points on the free boundary.  

To this end, we might as well assume \begin{equation*}\text{Null average assumption (NA):}
\sum u_j=0 \text{ pointwisely in $\Omega$}.
\end{equation*} This can be achieved by subtracting from each function the average $\frac{1}{N}\sum u_j$, which has no effect on the free boundaries.  Moreover, we focus on the free boundary with highest multiplicity, namely, $\cap_{j=1,2,\dots,N-1}\partial\{u_j>u_{j+1}\}$, as free boundaries with lower multiplicities involve fewer membranes. 

Although the method works for general forcing terms with some regularity, we assume the following simplification in order not to dilute our focus: \begin{equation*}\text{Constant forcing assumption (CF): }
\text{ The forcing terms $f_j$'s are constants}.
\end{equation*} 

We also assume that the forcing terms satisfy the following non-degeneracy condition, which in a sense corresponds to the strict concavity assumption on the obstacle in the classical obstacle problem:
$$\text{Non-degeneracy assumption (ND): }  f_j-f_{j+1}\ge\theta \text{ for some constant }  \theta>0.$$

For $p\in\cap_{j=1,2,\dots,N-1}\partial\{u_j>u_{j+1}\}$, we define, for $j=1,2,\dots, N$, the blow-up of $u_j$ at $p$ at scale $r>0$ to be $$\Tprj(x):=\frac{1}{r^2}u_j(p+rx).$$With $C^{1,1}$-regularity and the null average  assumption, these are compact families of functions in $C^{1,\alpha}_{loc}(\mathbb{R}^d)$ as $r\to 0$. The limits of subsequences are \textit{the blow-up profiles} or tangent solutions, and we use the notation $T_{p,0}u_j$ to denote a generic blow-up limit of $u_j$ at $p$.  This is a slight abuse of notation since the limit very much depends on the particular subsequence of $r\to 0$.  It is nevertheless useful in stating properties that hold true for all such limits. 

To classify all such limits, one needs to control the solutions at very fine scales. Inspired by the beautiful works of Weiss \cite{Wei1}\cite{Wei2}, our tool to achieve this control is  the following Weiss-type energy functional. At $p$ and scale $r$ it is defined as \begin{equation*}\Fpr:=\frac{1}{r^{d+2}}\int_{B_r(p)}\sum(|\nabla u_j|^2+2 f_ju_j)-\frac{2}{r^{d+3}}\int_{\partial B_r(p)}\sum u_j^2.\end{equation*}

Using the scaling symmetry of this functional and a radial competitor for the energy $\mathcal{F}$, we have the following monotonicity result: \begin{thm}
Under (CF), solutions to the $N$-membranes problem satisfy \begin{equation*}
\frac{d}{dr}|_{r=r_0}\Fpr\ge \frac{1}{r_0^{d+2}}\int_{\partial B_{r_0}(p)}\sum (\frac{\partial}{\partial r}u_j-2u_j/r_0)^2
\end{equation*} whenever $B_{r_0}(p)$ is compactly contained in $\Omega$. 
\end{thm} Here $\frac{\partial}{\partial r}$ is the derivative in the radial direction. 

This not only gives the monotonicity of the  Weiss-type energy along the radii, but also gives the condition for constancy, namely, $u_j$'s are all homogeneous of degree 2 with respect to $p$. Since blow-up solutions only see the Weiss-type energy at the {\em infinitesimal} scale, they have constant  Weiss-type energies and in particular one has the following \begin{thm}
Under (NA) and (CF), suppose $\{u_j\}_{j=1,2,\dots, N}$  solves the $N$-membranes problem and that $p\in\cap_{j=1,2,\dots,N-1}\partial\{u_j>u_{j+1}\}$. Then the blow-up limits satisfy 
$$T_{p,0}u_{j}(\lambda x)=\lambda^2T_{p,0}u_{j}(x)$$ for all $\lambda>0$ and $x\in\mathbb{R}^d$.
\end{thm} 

Even with this extra rigidity,  classifying all blow-up profiles for general spatial dimension seems out of reach at the moment. We therefore focus on the special case of the plane.

In this setting, we have the following complete classification of solutions that are homogeneous of degree 2: 

\begin{thm}
Let $\{u_j\}_{j=1,2,\dots N}$ be homogeneous of degree 2  connected solution to  the $N$-membranes problem in $\mathbb{R}^2$.  Under (NA), (CF), (ND),  and assume $0\in\cap_{j=1,2,\dots, N-1}\partial\{u_j>u_{j-1}\}$, one has

$$u_j(x)=P(x)+a_j(x\cdot e)_+^2-b_j(x\cdot e)_-^2,$$where $a_j$ and $b_j$ are real numbers, $e$ is a unit vector and $P$ is a quadratic polynomial. 

In particular, $\partial\{u_j>u_{j+1}\}$ coincides with a single line for all $j$. \end{thm} 

The precise definition of connected solutions is postponed to the final section of this work. It corresponds to the physical situation where there is forcing between the membranes when they contact. Since the problem reduces to one involving fewer membranes when \textit{connectedness} fails, this condition is not a restriction on the classification result as any solution can be split into connected solutions.

It is very interesting to note the perfect alignment of free boundaries. This agrees with the physical intuition that the upward forcing must balance with the downward.  Mathematically, this means that at the infinitesimal level, we do not see the `phantom jumps', which gives hope to the study of free boundary regularity in the plane. 

This paper is organised as follows: In the next section we prove the optimal interior regularity of solutions to $N$-membranes problem in general spatial dimensions, after which we establish monotonicity of the  Weiss-type energy functional  and some of its consequences. In the last section, we give the  classification of homogeneous solutions in the plane.

\section{Optimal interior regularity of solutions}

In this section we deal with the optimal interior $C^{1,1}$-regularity of solutions and prove Theorem 1.1. To be precise, by solutions of the $N$-membranes problem in $\Omega$ we mean 

\begin{defi}
Given $f_j\in C^{\alpha}(\Omega)$ for $j=1,2,\dots,N$, a solution to the $N$-membranes problem are $N$ functions $\{u_j\}_{j=1,2,\dots.N}$ in $$\mathcal{K}:=\{\{v_j\}|v_j\in H^{1}(\Omega), v_1\ge v_2\ge\dots\ge v_N \text{ in $\Omega$}\}$$such that $$\mathcal{F}(\{u_j\}):=\int_{\Omega}\sum( |\nabla u_j|^2+2f_ju_j)dx\le\int_{\Omega}\sum (|\nabla v_j|^2+2 f_jv_j)dx=:\mathcal{F}(\{v_j\})$$ for all $\{v_j\}\in\mathcal{K}$ and $v_j=u_j$ on $\partial \Omega$, $j=1,2,\dots, N$.
\end{defi} 

\begin{rem}
Under very mild conditions on $\Omega$ and boundary conditions, one can produce solutions by the direct method. Uniqueness is not an issue either due to the convexity of $\mathcal{F}$. We do not concern ourselves with these issues but focus solely on regularity estimates. To this end, one needs some regularity assumptions on the forcing terms $f_j$ that are better than mere boundedness. Here we choose to deal with $f_j$ in H\"older classes. 
\end{rem} 

A standard argument gives the Euler-Lagrange equations in the $\mathcal{L}^p$ sense \cite{CV}:
\begin{prop}
$$\Delta u_j=f_j \text{ in $\{u_{j-1}>u_{j}>u_{j+1}\}$},$$ and 
$$\Delta u_j=\Delta u_{j+1} \text{ in $\{u_j=u_{j+1}\}$}.$$
\end{prop} 
\begin{rem}
As a simple corollary, in $\{u_{k-1}<u_{k}=u_{k+1}=\dots=u_{k+m}<u_{k+m+1}\}$, $$\Delta(\sum_{j=k,k+1,\dots,k+m}u_j)=\sum_{j=k,k+1,\dots,k+m}f_j.$$

Also, $$\Delta(\sum_{j=1,2,\dots.N}u_j)=\sum_{j=1,2,\dots,N}f_j \text{ in $\Omega$}.$$
\end{rem} 

Due to the local nature of our estimates, we take $\Omega$ to be the unit ball and $K=B_{1/2}$.  Up to a scaling, we will also assume, in this section,$$\sum_{j=1}^N(\|u_j\|_{\mathcal{L}^{\infty}(B_1)}+\|f_j\|_{C^{\alpha}(B_1)})=1.$$

Under this normalisation,  elliptic regularity and Proposition 2.3 give: \begin{lem}
$$|\Delta u_j|\le C$$ and $$\|\sum u_j\|_{C^{2,\alpha}(B_{7/8})}\le C$$ for some universal $C$. 
\end{lem}

The next important step is the quadratic growth from contact points. It is essentially due to Lindgren and Razani \cite{LR}. Its proof is presented  here for the convenience of the reader.

\begin{prop}There is a universal constant $C$  such that 
for $j=1,2,\dots, N-1$, $p\in \{u_j=u_{j+1}\}\cap \overline{B_{3/4}}$ and $r<1/4$, one has 
$$\sup_{B_r(p)}(u_j-u_{j+1})\le Cr^2.$$
\end{prop} 

\begin{proof}
Due to the boundedness assumption on $u_j$, it suffices to find a  universal $C$ such that for each $r$, either $$\sup_{B_r(p)}(u_j-u_{j+1})\le Cr^2$$ or $$\sup_{B_r(p)}(u_j-u_{j+1})\le \frac{1}{4}\sup_{B_{2r}(p)}(u_j-u_{j+1}).$$

Suppose, on the contrary, such a $C$ does not exist. Then we find sequences of functions $\{u^{\nu}_j\}_{j=1,2,\dots, N}$ and $\{f^{\nu}_j\}_{j=1,2,\dots, N}$ for $\nu\in\mathbb{N}$, where $\{u^{\nu}_j\}_{j=1,2,\dots, N}\in\mathcal{K}$ minimises the energy $\mathcal{F}$ with $\{f^{\nu}_j\}_{j=1,2,\dots, N}$ as forcing terms and satisfies $$\sum_{j=1}^N(\|u_j^\nu\|_{\mathcal{L}^{\infty}(B_1)}+\|f_j^\nu\|_{C^{\alpha}(B_1)})=1,$$ but  \begin{equation*}\sup_{B_{r_\nu}(p_\nu)}(u^\nu_{j_\nu}-u^\nu_{j_{\nu}+1})> \nu r_{\nu}^2\end{equation*}
and \begin{equation*}\sup_{B_{r_\nu}(p_\nu)}(u^\nu_{j_\nu}-u^\nu_{j_{\nu}+1})>\frac{1}{4}\sup_{B_{2r_\nu}(p_\nu)}(u^\nu_{j_\nu}-u^\nu_{j_{\nu}+1})\end{equation*}
for some  $j_\nu\in\{1,2,\dots,N-1\}$, $p_\nu\in\{u^\nu_{j_\nu}=u^\nu_{j_{\nu}+1}\}$ and $r_\nu>0$. 

Boundedness of $u_j^\nu$ and $\nu\to\infty$ imply $r_\nu\to 0$. Also, up to a subsequence, we might assume $j_{\nu}=j_1$ for all $\nu$. 

Define $S_\nu=\sup_{B_{r_\nu}(p_\nu)}(u^\nu_{j_1}-u^\nu_{j_{1}+1})$, and the function $$h^\nu(x)=\frac{1}{S_\nu}(u^\nu_{j_1}(p_\nu+r_\nu x)-u^\nu_{j_{1}+1}(p_\nu+r_\nu x)).$$Then $$h^\nu(0)=0,$$ $$\sup_{B_1}h^\nu=1,$$ $$0\le h^\nu< 4 \text{ in $B_2$,}$$ and $$|\Delta h^\nu|=\frac{(r_{\nu})^2}{S_\nu}|\Delta u_{j_1}^\nu-\Delta u_{j_1+1}^\nu|\le \frac{C}{\nu} \text{ in $B_2$}.$$ Here the constant $C$ is the one from Lemma 2.5. 

Standard elliptic regularity then gives a subsequence $h^\nu\to h$ in $C^{1,\alpha}(B_1)$. The limit $h$ satisfies $\Delta h=0$ and $h\ge 0$ in $B_1$. Also $h(0)=0$ and $\sup_{B_1}h=1$, which contradicts the strong maximum principle for harmonic functions. 
\end{proof} 

\begin{rem}
{\em Heuristically}, at such a contact point $p$, $C^{1,\alpha}$-regularity of $\{u_j\}$ gives $u_j(p)=u_{j+1}(p)$ and $\nabla u_j(p)=\nabla u_{j+1}(p).$ $u_j\ge u_{j+1}$ then forces $D^{2}u_j(p)\ge D^{2}u_{j+1}(p)$, while the previous proposition gives $D^2u_j(p)\le D^2u_{j+1}(p)+C.$

This would give very strong implications at contact points of the highest multiplicity, namely,  $p\in\cap_{j=1,2,\dots, N-1}\{u_{j}=u_{j+1}\}\cap B_{1/2}$ since there one has $$|D^{2}u_j(p)-D^2u_{j+1}(p)|\le C \text{ for all $j\le N-1$}.$$ This combined with the fact from Lemma 2.5 $$|\sum D^2u_j(p)|\le C$$ gives $|D^2u_j(p)|\le C$ for all $j$, which is the desired $C^{1,1}$ estimate at $p\in\cap_{j=1,2,\dots, N-1}\{u_{j}=u_{j+1}\}\cap B_{1/2}$.

For points outside the full contact region, in a neighbourhood one sees fewer membranes, this allows a proof  by induction on the number of membranes $N$. The base case $N=2$ is dealt with in \cite{Sil}.\end{rem}

\begin{prop}
For $p\in B_{1/2}$, $|D^2u_j(p)|\le C$ for some universal $C$. 
\end{prop} 

\begin{proof}
Depending on the contacting situation between membranes, we divide $\{1,2,\dots,N\}$ into groups, so that membranes within the same group contact at $p$, while different groups are strictly ordered. That is, we find some $\gamma\in \{1,2,\dots,N\}$ and integers $$j_0=1\le j_1<j_2<\cdots<j_{\gamma-1}<j_{\gamma}=N$$ such that $$u_j(p)=u_{j+1}(p) \text{ for each $j_{k}+1\le j\le j_{k+1}$,}$$ and $$u_{j_k}(p)> u_{j_k+1}(p) \text{ for each $k=1,2,\dots,\gamma-1$}.$$

If it happens that $\gamma=1$, namely, all membranes touch at $p$,  then Remark 2.7 gives the desired estimate. Thus it suffices to deal with the case where $\gamma>1$.

Within each group, Proposition 2.6 gives \begin{equation}
0\le u_{j}-u_{j+1}\le CR^2 \text{ in $B_{R}(p)$ for  $j_{k}+1\le j\le j_{k+1}.$}
\end{equation} 

For each $k=1,2,\dots,\gamma-1$, define $$R_k=dist(p,\partial\{u_{j_k}>u_{j_k+1}\}).$$ Then in $B_{R_k}(p)$, $u_{j_k}>u_{j_k+1}$, and there exists some $y_k\in\partial B_{R_k}(p)$ such that $u_{j_k}(y_k)=u_{j_k+1}(y_k)$. Proposition 2.6 gives the estimate between groups $$0\le u_{j_k}-u_{j_k+1}\le CR_k^2 \text{ in $B_{R_k}(p)$.}$$

Pick $k^*$ such that $R_{k^*}\ge R_k$ for all $k$, then for all $k=1,2,\dots,\gamma-1$, $$0\le u_{j_k}-u_{j_k+1}\le CR_{k^*}^2 \text{ in $B_{R_{k^*}}(p)$.}$$ Combining with the estimate within the same group (2.1), one obtains \begin{equation}
0\le u_j-u_{j+1}\le CR_{k^*}^2 \text{ for all $j$}.
\end{equation}

Now for $s\le j_{k^*}$, define $$v_s(x):=\frac{1}{R_{k^*}^2}(u_s(p+R_{k^*}x)-u_{j_{k^*}}(p+R_{k^*}x)).$$ Then $\{v_s\}_{s=1,\dots, j_{k^*}-1}$ solves the multiple membranes problem in $B_1$ with forcing terms $(f_s(p+R_{k^*}\cdot)-f_{j_{k^*}}(p+R_{k^*}\cdot))$.  

The fact $R_{k^*}\le 1$ implies that the $C^{\alpha}$ norm of such forcing terms is less than the sum of $C^{\alpha}$ norms of $f_s$ and of $f_{j_{k^*}}$, which is universally bounded.  Meanwhile, (2.2) gives $0\le v_s\le C$ in $B_1$.  As a result, the induction hypothesis on $N$, applied to $\{v_s\}$ gives $$|D^2v_s|(0)\le C,$$ which in turn implies \begin{equation*}
|D^2u_s(p)-D^2u_{j_{k^*}}(p)|\le C \text{ for all $s< j_{k^*}$}. 
\end{equation*} 

A symmetric argument applies to $s>  j_{k^*}$ and gives \begin{equation*}
|D^2u_s(p)-D^2u_{j_{k^*}}(p)|\le C \text{ for all $s> j_{k^*}$}. 
\end{equation*} 

Combining these with $|\sum D^2u_j(p)|\le C$ gives $$|D^2u_{j_k^*}(p)|\le C.$$ This in turn implies $$|D^2u_j(p)|\le C \text{ for all $j$}.$$ \end{proof}

\section{ Weiss-type monotonicity formula and homogeneity of blow-up solutions}
Starting from this section we turn our attention to the free boundary, and in particular, the piece of {\em the highest multiplicity}, namely, points in  $\cap_{j=1,2,\dots, N-1}\partial\{u_j>u_{j+1}\}$. This piece captures the essential difficulty as explained in the Introduction. To do this, it does no harm to subtract the average of the membranes and as a result to assume (NA).  We also assume (CF) to concentrate on the main ideas. 

Under these assumptions, the first result in this section is the monotonicity of the  Weiss-type energy as in Theorem 1.2. To be precise, we are dealing with the functional:
\begin{defi}
For $u_j\in C^{1}(\Omega)$ and $f_j\in\mathcal{L}^\infty(\Omega)$ $j=1,2,\dots, N$, $p\in\Omega$ and $0<r<dist(p,\partial\Omega)$, the Weiss-type energy of $\{u_j\}_{j=1,2,\dots, N}$ at $p$ and scale $r$ is $$\Fpr=\EFpr.$$
\end{defi} 

\begin{rem}
Since we have already established interior regularity, this functional is well-defined for our solutions. Moreover, it is not too difficult to prove that this functional is differentiable in $r$. \end{rem}
 
Now we begin the proof for Theorem 1.2, translation-invariance of the functional allows us to consider the following:

\begin{prop}
Solutions to the $N$-membranes problem in $B_1$ under (CF) satisfies $$\frac{d}{dr}|_{r=r_0}\For \ge \frac{1}{r_0^{d+2}}\int_{\partial B_{r_0}(0)}\sum(\frac{\partial}{\partial r}u_j-2u_j/r_0)^2$$ for all $r_0\le 1$.  
\end{prop} Again $\frac{\partial}{\partial r}$ denotes the derivative in the radial direction.

This would be the consequence of the following lemmata.

\begin{lem}
Under the assumptions in Proposition 3.1, one has $$(d+2)\int_{B_1}\sum(|\nabla u_j|^2+2f_ju_j)\le \int_{\partial B_1}\sum(|\TDer u_j|^2+2f_ju_j+4 u_j^2).$$
\end{lem} Here $\TDer u$ is the tangential derivative of a function $u$. 
\begin{proof}
For a small $\epsilon>0$, we construct the following competitors for the energy $\mathcal{F}$. 

For $j=1,2,\dots, N$, 
$$v_j(x):=\begin{cases} |x|^2u_j(\frac{x}{|x|}) &\text{ for $|x|> 1-\epsilon$,}\\(1-\epsilon)^2u_j(\frac{x}{1-\epsilon}) &\text{ for $|x|\le 1-\epsilon$.}\end{cases}$$

Direct computations yield $$|\nabla v_j|^2(x)=\begin{cases} 4|x|^2u_j^2(\frac{x}{|x|})+|x|^2|\nabla_\tau u_j|^2(\frac{x}{|x|}) &\text{ for $|x|> 1-\epsilon$,}\\(1-\epsilon)^2|\nabla u_j|^2(\frac{x}{1-\epsilon}) &\text{ for $|x|\le 1-\epsilon$.}\end{cases}$$

Therefore, \begin{align*}
\int_{B_{1-\epsilon}}|\nabla v_j|^2&=(1-\epsilon)^2\int_{B_{1-\epsilon}}|\nabla u_j|^2(\frac{x}{1-\epsilon})\\ &=(1-\epsilon)^{d+2}\int_{B_1}|\nabla u_j|^2.
\end{align*}

\begin{align*}
\int_{B_1\backslash B_{1-\epsilon}}|\nabla v_j|^2&=\int_{1-\epsilon}^1ds\int_{\partial B_s}4s^2u_j^2(\frac{x}{|x|})+s^2|\TDer u_j|^2(\frac{x}{|x|})\\&=\int_{1-\epsilon}^1ds\int_{\partial B_1}4s^{d+1}u_j^2+s^{d+1}|\TDer u_j|^2\\&=\frac{4}{d+2}(1-(1-\epsilon)^{d+2})\int_{\partial B_1}u_j^2+\frac{1}{d+2}(1-(1-\epsilon)^{d+2})\int_{\partial B_1}|\TDer u_j|^2.
\end{align*}

Moreover, \begin{align*}
\int_{B_{1-\epsilon}}f_jv_j&=\int_{B_{1-\epsilon}}(1-\epsilon)^2f_ju_j(\frac{x}{1-\epsilon})\\&=(1-\epsilon)^{d+2}\int_{B_1}f_ju_j.
\end{align*}

And \begin{align*}
\int_{B_1\backslash B_{1-\epsilon}}f_jv_j&=f_j\int_{B_1\backslash B_{1-\epsilon}}|x|^2u_j(\frac{x}{|x|})\\&=f_j\int_{1-\epsilon}^1ds\int_{\partial B_s}s^2u_j(\frac{x}{|x|})\\&=f_j\int_{1-\epsilon}^1s^{d+1}ds\int_{\partial B_1}u_j\\&=\frac{1}{d+2}(1-(1-\epsilon)^{d+2})f_j\int_{\partial B_1}u_j.
\end{align*}
 
Here we have used the fact that $f_j$'s are constants. 

Combining all these gives the energy for $v_j$. Meanwhile, since $v_j\ge v_{j+1}$ in $B_1$ and $v_j=u_j$ along $\partial B_1$, these are admissible in the minimisation of $\mathcal{F}$. Minimality of $u_j$ gives \begin{align*}
\int_{B_1}\sum(|\nabla u_j|^2+2f_ju_j)\le & (1-\epsilon)^{d+2}\int_{B_1}\sum(|\nabla u_j|^2+2f_ju_j)+\frac{4}{d+2}(1-(1-\epsilon)^{d+2})\int_{\partial B_1}\sum u_j^2\\+\frac{1}{d+2}&(1-(1-\epsilon)^{d+2})\int_{\partial B_1}\sum|\TDer u_j|^2+\frac{1}{d+2}(1-(1-\epsilon)^{d+2})\int_{\partial B_1}\sum 2f_ju_j.
\end{align*}

This being true for all small $\epsilon$, we can take the linear order terms in $\epsilon$:
$$(d+2)\int_{B_1}\sum(|\nabla u_j|^2+2f_ju_j)\le \int_{\partial B_1}\sum|\TDer u_j|^2+\int_{\partial B_1}\sum 2f_ju_j+4\int_{\partial B_1}\sum u_j^2.$$
\end{proof} 

This lemma in turn implies the desired monotonicity at $r=1$:

\begin{lem}
Under the assumptions of Proposition 3.1, $$\frac{d}{dr}|_{r=1}\For\ge \int_{\partial B_{1}(0)}\sum(\frac{\partial}{\partial r}u_j-2u_j)^2.$$
\end{lem} 

\begin{proof}
The  Weiss-type energy functional naturally splits into two pieces $$E(r)=\frac{1}{r^{d+2}}\int_{B_r}\sum(|\nabla u_j|^2+2f_ju_j),$$ and $$F(r)=\frac{2}{r^{d+3}}\int_{\partial B_r}\sum u_j^2.$$

We differentiate each piece separately. 

\begin{align*}
\frac{d}{dr} E(r)&=\frac{-(d+2)}{r^{d+3}}\int_{B_r}\sum(|\nabla u_j|^2+2f_ju_j)+\frac{1}{r^{d+2}}\int_{\partial B_r}\sum(|\nabla u_j|^2+2f_ju_j)\\&=\frac{-(d+2)}{r^{d+3}}\int_{B_r}\sum(|\nabla u_j|^2+2f_ju_j)+\frac{1}{r^{d+2}}\int_{\partial B_r}\sum(|\nabla_\tau u_j|^2+2f_ju_j)+\frac{1}{r^{d+2}}\int_{\partial B_r}\sum|\RDer u_j|^2.
\end{align*} At $r=1$, one has \begin{align*}\frac{d}{dr}|_{r=1} E(r)&=-(d+2)\int_{B_1}\sum(|\nabla u_j|^2+2f_ju_j)+\int_{\partial B_1}\sum(|\nabla_\tau u_j|^2+2f_ju_j)+\int_{\partial B_1}\sum|\RDer u_j|^2\\&\ge\int_{\partial B_1}\sum|\RDer u_j|^2-4\int_{\partial B_1}\sum u_j^2\end{align*} by Lemma 3.4.

Meanwhile, \begin{align*}
\frac{d}{dr} F(r)&=\frac{d}{dr} (\frac{2}{r^{4}}\int_{\partial B_1}\sum u^2_j(r\cdot))\\&=\frac{-8}{r^5}\int_{\partial B_1}\sum u^2_j(r\cdot)+\frac{2}{r^4}\int_{\partial B_1}\sum 2u_j(r\cdot)\RDer u_j(r\cdot).
\end{align*} Again plug in $r=1$ gives
$$\frac{d}{dr}|_{r=1} F(r)=-8\int_{\partial B_1}\sum u^2_j+2\int_{\partial B_1}\sum 2u_j\RDer u_j.$$

To conclude, simply note \begin{align*}
\frac{d}{dr}|_{r=1}\For&=\frac{d}{dr}|_{r=1} E(r)-\frac{d}{dr}|_{r=1} F(r)\\&\ge \int_{\partial B_1}\sum|\RDer u_j|^2-4\int_{\partial B_1}\sum u_j^2+8\int_{\partial B_1}\sum u^2_j-2\int_{\partial B_1}\sum 2u_j\RDer u_j\\&=\int_{\partial B_1}\sum|\RDer u_j|^2-4\int_{\partial B_1}\sum u_j\RDer u_j+4\int_{\partial B_1}\sum u_j^2\\&=\int_{\partial B_{1}}\sum(\frac{\partial}{\partial r}u_j-2u_j)^2.
\end{align*}\end{proof} 

Now to finish the proof for Proposition 3.1, we use the following symmetry of the  Weiss-type energy: \begin{equation}
\For=W(\{\Torj\},0,1).
\end{equation} And in particular, \begin{align*}
\frac{W(\{u_j\},0,r+h)-\For}{h}&=\frac{W(\{\Torj\},0,1+h/r)-W(\{\Torj\},0,1)}{h}\\&=\frac{W(\{\Torj\},0,1+h/r)-W(\{\Torj\},0,1)}{h/r}\frac{1}{r}.
\end{align*}Sending $h\to 0$ gives $$\frac{d}{dr}\For=\frac{1}{r}\frac{d}{ds}|_{s=1}W(\{\Torj\},0,s)\ge\frac{1}{r}\int_{\partial B_1}\sum(\RDer \Torj-2\Torj)^2,$$ from which the desired estimate follows by a change of variable. 

This concludes the proof for Proposition 3.1, which implies Theorem 1.2.

As a simple corollary of this monotonicity, we obtain the homogeneity for solutions with constant  Weiss-type energy:
\begin{cor}
Under the assumptions of Proposition 3.1, if $\For$ is constant in $r$, then $u_j$'s are homogeneous of degree 2. 
\end{cor}
\begin{proof}
The constancy condition implies $(\frac{\partial}{\partial r}u_j-2u_j/r)^2=0$ a.e., that is, $\frac{d}{d\lambda}\frac{u_j(\lambda x)}{\lambda^2}=0.$
\end{proof} 

With these, we are in the place to study blow-up solutions at free boundary points of the highest multiplicity.

\begin{defi}
Suppose $\{u_j\}$ solves the $N$-membranes problem in $\Omega$ and $p\in\cap_{j=1,2,\dots,N-1}\partial\{u_j>u_{j+1}\}\cap\Omega$, the blow-up at $p$ at scale $r>0$ is 
$$\Tprj(x):=\frac{1}{r^2}u_j(p+rx).$$
\end{defi} 

\begin{rem}
It should be clear that for each $r>0$, $\{\Tprj\}$ is a minimiser of $\mathcal{F}$ in $B_{R}$ for all $R<dist(p,\partial\Omega)/r$.
\end{rem} 
The first observation is their compactness in $C^{1,\alpha}_{loc}(\mathbb{R}^d)$:
\begin{prop}
For $p\in\cap_{j=1,2,\dots,N-1}\partial\{u_j>u_{j+1}\}\cap\Omega$, up to a subsequence of $r\to 0$, $\{\Tprj\}$ converges in $C^{1,\alpha}_{loc}(\mathbb{R}^d)$ to $\{T_{p,0}u_j\}$, a local minimiser of $\mathcal{F}$.
\end{prop} 

\begin{proof}
Denote $dist(p,\partial\Omega)>0$ by $D$, then in $B_{D/r}$, $\Tprj(x)=\frac{1}{r^2}u_j(p+rx)$, $\nabla\Tprj(x)=\frac{1}{r}\nabla u_j(p+rx)$ and $D^2\Tprj(x)=D^2u_j(p+rx)$. In particular, if we apply a scaled version of the $C^{1,1}$-estimate, we have $$|D^2\Tprj|\le C \text{ in $B_{D/(2r)}$}$$ for some $C$ depending on $d, N$ and $\sum\|u_j\|_{\mathcal{L}^\infty(\Omega)}+\sum\|f_j\|_{C^\alpha(\Omega)}.$

Meanwhile, with the contacting condition and the null average assumption in effect, we have $\Tprj(0)=0$, and $\nabla\Tprj(0)=0$. These combined with the bound on Hessian gives a uniform bound on $\Tprj$ in $C^{1,1}(B_R)$ for any fixed $R>0$, which gives, up to a subsequence, convergence of $\Tprj$ in $C^{1,\alpha}(B_R)$ for any fixed $R>0$. The convergence in $C_{loc}^{1,\alpha}(\mathbb{R}^d)$ follows by a diagonal argument. 

Now for a positive $R>0$, we show that the limit $\{T_{p,0}u_j\}$, as $r_k\to 0$,  is a minimiser for $\mathcal{F}$ in $B_R$. 

Suppose not, we find $\{w_j\}$ such that $w_j\ge w_{j+1}$ and $w_j=T_{p,0}u_j$ along $\partial B_R$ but $$\int_{B_R}\sum(|\nabla w_j|^2+2f_jw_j)<\int_{B_R}\sum(|\nabla T_{p,0}u_j|^2+2f_jT_{p,0}u_j)-\delta$$ for some $\delta>0$. By uniform $C^{1,\alpha}$ convergence in $B_{R}$, this gives $$\int_{B_{R}}\sum(|\nabla w_j|^2+2f_jw_j)<\int_{B_{R}}\sum(|\nabla T_{p,r_k}u_j|^2+2f_jT_{p,r_k}u_j)-\frac{1}{2}\delta$$ for all $k>K(\delta)$.

Define $v_j$ for $j=1,2,\dots, N$ by $$v_j(x)=\begin{cases} w_j(x) &\text{ for $|x|\le R$,}\\ T_{p,0}u_j(x) &\text{ for $R<|x|\le 2R$.}\end{cases}$$ Take a cut-off function $\eta$ that is supported in $B_{2R}$ and equals $1$ in $B_R$, then $\tilde v_j:=\eta v_j+(1-\eta)T_{p,r_k}u_j$ agrees with $T_{p,r_k}u_j$ on $\partial B_{2R}$, and satisfies $\tilde{v}_{j}\ge \tilde{v}_{j+1}$. Therefore, minimality of $\{T_{p,r_k}u_j\}_{j=1,2,\dots, N}$ gives 
\begin{align*}&\int_{B_{2R}}\sum(|\nabla T_{p,r_k}u_j|^2+2f_jT_{p,r_k}u_j)\le \int_{B_{2R}}\sum(|\nabla \tilde{v}_j|^2+2f_j\tilde{v}_j)\\&=\int_{B_R}\sum(|\nabla w_j|^2+2f_jw_j)+\int_{B_{2R}\backslash B_R}\sum(|\nabla \tilde{v}_j|^2+2f_j\tilde{v}_j)\\&<\int_{B_{R}}\sum(|\nabla T_{p,r_k}u_j|^2+2f_jT_{p,r_k}u_j)-\frac{1}{2}\delta+\int_{B_{2R}\backslash B_R}\sum(|\nabla T_{p,r_k}u_j|^2+2f_jT_{p,r_k}u_j)\\&+\int_{B_{2R}\backslash B_R}\sum(|\nabla [\eta (T_{p,r_k}u_j-T_{p,0}u_j)|^2+2f_j[\eta(T_{p,r_k}u_j-T_{p,0}u_j)]),
\end{align*}that is, $$\int_{B_{2R}\backslash B_R}\sum(|\nabla [\eta (T_{p,r_k}u_j-T_{p,0}u_j)|^2+2f_j[\eta(T_{p,r_k}u_j-T_{p,0}u_j)])>\frac{1}{2}\delta$$ for all $k>K(\delta).$

Note that the left-hand side converges to $0$ as $k\to\infty$ by uniform $C^{1,\alpha}$ convergence in $B_{2R}$, leading to a contradiction. \end{proof}

With all these preparations, we can finally give the proof of Theorem 1.3:
\begin{proof}
Fix some $R>0$, then uniform $C^{1,\alpha}$ convergence and the symmetry property of the  Weiss-type energy give\begin{align*}
W(\{T_{p,0}u_j\}, 0, R)&=\lim W(\{T_{p,r}u_j\}, 0, R)\\&=\lim W(\{u_j\}, 0, rR)\\&=W(\{u_j\}, 0, 0+).
\end{align*}The last limit exists by monotonicity. 

In particular, the value of $W(\{T_{p,0}u_j\}, 0, R)$ is independent of $R$. Now being a local minimiser, Corollary 3.4 applies to $\{T_{p,0}u_j\}$ and gives the homogeneity. 
\end{proof} 

\begin{rem}
The fact that $f_j$'s are constant gives scaling symmetry to the  Weiss-type energy, which makes the proof of Proposition 3.1 particularly straightforward. However, even for variable forcing terms, a statement like the one in Theorem 1.3 would still be true, as long as we have some regularity of the $f_j$'s. This is because even if $f_j$'s are not constant, the blow-up solutions would only see the value of $f_j$'s at the origin. As a result, the argument for blow up homogeneity still applies. 
\end{rem} 

Finally we observe the  non-degeneracy of blow-up solutions. This in particular says that $0$ is a free boundary point of the highest multiplicity for blow-up profiles. 

\begin{prop}
For $p\in\cap_{j=1,2,\dots,N-1}\partial\{u_j>u_{j+1}\}$ and under (ND), $$\sup_{B_1}(T_{p,0}u_j-T_{p,0}u_{j+1})\ge C \text{ for all $j$,}$$ where $C$ is some constants depending on $d, N$ and $\theta$.  
\end{prop} 

\begin{proof}
Simply note that \begin{align*}\sup_{B_1}(T_{p,0}u_j-T_{p,0}u_{j+1})&=\lim(\sup_{B_1}(T_{p,r}u_j-T_{p,r}u_{j+1}))\\&=\lim\frac{1}{r^2}\sup_{B_r(p)}(u_j-u_{j+1}).\end{align*}

Now as in the classical obstacle problem, the non-degeneracy condition $f_j-f_{j+1}\ge\theta>0$ forces $$\sup_{B_r(p)}(u_j-u_{j+1})\ge C(n,\theta)r^2.$$ For a detailed proof of this fact, the reader should consult Proposition 2 in Lindgren-Razani \cite{LR}.\end{proof}

\section{Homogeneous solutions in the plane}
The next step in the study of $\cap_{j=1,2,\dots, N-1}\partial\{u_{j}>u_{j+1}\}$ would be to classify all possible blow-up profiles. We take this task in the plane $\mathbb{R}^2$ under assumptions (NA), (CF), (ND) and that $0\in\cap_{j=1,2,\dots, N-1}\partial\{u_{j}>u_{j+1}\}$.  

In this section, we use $\{u_j\}_{j=1,2,\dots, N}$, instead of $\{T_{0,r}u_j\}_{j=1,2,\dots, N}$, to denote the blow-up profile. We further make the following distinction between two types of solutions depending on the contact sets on the circle:

\begin{defi}
For $k=1,2,\dots, N-1$, $u_k$ and $u_{k+1}$ are \textit{disconnected} if $\{u_k=u_{k+1}\}\cap\mathbb{S}^1$ consists of isolated points. Otherwise, we say they are \textit{connected}.\end{defi} 

Whenever $u_k$ and $u_{k+1}$ are disconnected, there is no interaction between the two membranes in the sense that their Laplacians experience no jumps due to their contact. In this case, the two groups of membranes, $\{u_j\}_{j\le k}$ and $\{u_j\}_{j\ge k+1}$, solve two multiple-membranes problems separately. As a result, to classify all possible blow-up profiles, it suffices to consider the case when all consecutive membranes are connected and prove Theorem 1.4. 

To do so, we introduce $(\rho,\omega)$ as the variables in polar coordinates.  Also, for some interval on the circle $(\alpha,\beta)$, $\Gamma_{(\alpha,\beta)}$ denotes the cone generated by the interval, namely, $$\Gamma_{(\alpha,\beta)}=\{(\rho,\omega)|0<\rho<1,\alpha<\omega<\beta\}.  $$

The core of the argument is the following elementary observation:

\begin{lem}
Suppose $g$ is a non-negative, non-decreasing step function on $[0,\alpha]$ for some $0<\alpha\le\pi$, and that $v$ is a homogeneous of degree 2 function in $\Gamma_{[0,\alpha]}$ satisfying $$\Delta v(\rho,\omega)=g(\omega)$$ in the interior of the cone. 

If $v(\rho,0)=0$ and $\nabla v(\rho,0)=0$, then $v$ is convex in $\Gamma_{[0,\alpha]}$.
\end{lem} 

\begin{proof}
For some $0=\alpha_0<\alpha_1<\dots<\alpha_{s-1}<\alpha_s=\alpha$, we have $g=\sigma_j$ on $(\alpha_j-1,\alpha_j]$ and $\sigma_j\le\sigma_{j+1}.$

Then the initial conditions along the ray $\omega=0$ and $\alpha\le\pi$ imply that $v$ is given explicitly by $$v(x)=\frac{\sigma_0}{2}(x\cdot e_0)_+^2+\sum\frac{\sigma_j-\sigma_{j-1}}{2}(x\cdot e_{j})_+^2,$$where $e_j$ denotes the unit vector $\omega=\alpha_j+\pi/2.$

Being a combination of convex functions with non-negative coefficients, $v$ is convex.\end{proof}

\begin{cor}
Let $g:[0,\alpha]\to\mathbb{R}$ be a non-negative step function that is non-decreasing on $[0,\gamma]$ and non-increasing on $[\gamma,\alpha]$ for some $0\le\gamma\le\alpha\le 2\pi.$ 

If there is a nontrivial homogeneous of degree 2 function $v$ in $\Gamma_{[0,\alpha]}$, solving $$\Delta v=g$$ in the interior of the cone, and satisfying $v(\rho,0)=v(\rho,\alpha)=0$ and $\nabla v(\rho,0)=\nabla v(\rho,\alpha)=0$, then $$\alpha\ge\pi.$$

Furthermore, if $\alpha=\pi$, then $g$ must be constant. 
\end{cor}

\begin{proof}
Suppose $\alpha\le\pi$, an application of the previous lemma implies the convexity of $v$ in $\Gamma_{[0,\gamma]}$ and $\Gamma_{[\gamma,\alpha]}$.

\textit{Claim:  $v$ depends only on the $x_2$-variable and $g$ must be constant.} To see this, note that $\frac{\partial}{\partial x_1}v$ is homogeneous of degree 1. Therefore, if it is not zero along $\omega=\gamma$, it must have a sign along that ray.  

If $\frac{\partial}{\partial x_1}v>0$ along $\omega=\gamma$, then by convexity in $\Gamma_{[\gamma, 0]},$ $\frac{\partial}{\partial x_1}v>0$ along $\omega=0$, a contradiction. The case $\frac{\partial}{\partial x_1}v<0$ along $\omega=\gamma$ is similarly ruled out by convexity in $\Gamma_{[\gamma, 0]}.$ As a result, $\frac{\partial}{\partial x_1}v=0$ along $\omega=\gamma$. In this case, convexity implies that $\frac{\partial}{\partial x_1}v$ must vanish in $\Gamma_{[0,\alpha]}$. Therefore, $v$ depends only on $x_2$ in the cone, concluding the proof for Claim. 

This shows that $g=\frac{\partial^2}{\partial x_2^2}v$ must be constant if $\alpha\le\pi.$ 

 Now that $g$ is constant, $v$ is convex in the entire $\Gamma_{[0,\alpha]}$ by the previous lemma. If $\alpha<\pi$, then $v$ must be $0$, since both itself and its gradient vanish along the boundary of this strictly convex cone, contradicting our assumption that $v$ is nontrivial. \end{proof} 

A simple application of the maximum principle implies that each connected component of $\{u_j>u_{j+1}\}\cap\mathbb{S}^1$ has length bounded from below by a universal constant. In particular, there are finitely many of such components.  Our first goal is to show that for connected solutions,  $\{u_j>u_{j+1}\}\cap\mathbb{S}^1$ consists of exactly of one component of length exactly $\pi$. 

To apply the previous corollary, we first focus on open intervals $I\subset\mathbb{S}$ with the following property:

$$\text{P: For some $k$, $u_k>u_{k+1}$ on $I$, $u_k=u_{k+1}$ on $\partial I$, and $\Delta u_k$ is constant on $I$.  }$$

Note that on an interval with property $P$, any change of $\Delta(u_k-u_{k+1})$ is due to the change of $u_{k+1}$, that is, due to detaching/ contacting of $u_{k+1}$ with $u_{k+s}$ for $s\ge 2$.

\begin{prop}
If $I$ has Property P, then $|I|\ge\pi.$
\end{prop} 

\begin{proof}
Let $I$ be the shortest among intervals with Property P. It suffices to show $|I|\ge\pi$. 

Suppose $I$ is  a component of $\{u_k>u_{k+1}\}\cap\mathbb{S}^1.$

\textit{Claim: $\{u_{k+1}=u_{k+s}\}\cap I$ is connected for each $s\ge 2$.} Firstly, if $\{u_{k+1}=u_{k+2}\}\cap I$ is not connected, then there would be a component of $\{u_{k+1}>u_{k+2}\}$ \textit{strictly} contained in $I$. On this component, $\Delta u_{k+1}$ is constant since $u_k>u_{k+1}>u_{k+2}$ on this interval. Thus this component of $\{u_{k+1}>u_{k+2}\}$ has Property P, contradicting the minimality of $I$. 

If $\{u_{k+1}=u_{k+3}\}\cap I=\{u_{k+1}=u_{k+2}\}\cap \{u_{k+2}=u_{k+3}\}\cap I$ is not connected, then there would be a component of $\{u_{k+2}>u_{k+3}\}$ \textit{strictly} contained in $\{u_{k+1}=u_{k+2}\}\cap I$. On this component, $\Delta u_{k+2}$ is constant since $u_{k+1}=u_{k+2}>u_{k+3}$ on this interval and $\Delta u_{k+1}$ is constant there. Thus this component of $\{u_{k+2}>u_{k+3}\}$ has Property P, again contradicting the minimality of $I$. 

From here it is obvious how to inductively repeat the argument and conclude Claim. 

Note that by assumption $f_1\ge f_2\ge \dots\ge f_N$ and Remark 2.4, the more membranes are in contact with $u_{k+1}$, the smaller becomes $\Delta u_{k+1}.$ Since on $I$, the sets $\{u_{k+1}=u_{k+2}\}\supset\{u_{k+1}=u_{k+3}\}\supset\{u_{k+1}=u_{k+4}\}\supset\dots$ are all connected, $\Delta u_{k+1}$ would be non-increasing on a subinterval of $I$ and then non-decreasing in the rest of $I$. 

Take $v=u_k-u_{k+1}$, then we are in the situation of Corollary 4.3, which gives $|I|\ge \pi$.\end{proof} 

\begin{prop}
$\{u_1>u_2\}\cap\mathbb{S}^1$ consists of a single component of length exactly $\pi$. 
\end{prop} 

\begin{proof}
By Proposition 3.11, $\{u_1>u_2\}\cap\mathbb{S}^1$ is nonempty. Meanwhile, since on $\{u_1>u_2\}$, $\Delta u_1=f_1$, each component of $\{u_1>u_2\}\cap\mathbb{S}^1$ has Property P, and hence is of length no less than $\pi.$

In particular, there is exactly one component of  $\{u_1>u_2\}\cap\mathbb{S}^1$, for otherwise $\{u_1>u_2\}$ would consists of two components, each of length $\pi$, contradicting the connectedness of $u_1$ and $u_2$.  

Denote this single component of $\{u_1>u_2\}$ by $I$. We just need to show $|I|\le \pi$. 

Suppose, on the contrary, that $|I|>\pi.$ Then the previous proposition implies that  $I^c$ cannot contain any interval with Property P, and a similar argument as in its proof gives the connectedness of $\{u_2=u_s\}\cap I^c$ for all $s\ge 3$, which shows $\Delta u_2$ is non-increasing on a subinterval of $I^c$, and then non-decreasing in the rest of $I^c$. 

Now $u_1$ is homogeneous and has constant Laplacian in $I$,  it is a quadratic polynomial in $\Gamma_I$. By taking the same polynomial, we get a function $\tilde{u}_1$ in $I^c$ with the same Laplacian.  Moreover, $\tilde{u}_1$ and its gradient agree with $u_1$ and $\nabla u_1$ on $\partial\Gamma_{I^c}$.

In particular, if we take $v=\tilde{u}_1-u_2$, then Corollary 4.3 applies and gives $|I^c|\ge\pi$. \end{proof} 

With all these preparations, we can finally give 
\begin{proof} of Theorem 1.4.

Now that $\{u_1>u_2\}\cap\mathbb{S}^1$ consists of a single component $I$ of length $\pi$, a similar argument as in the proof of Proposition 4.4 shows $\Delta u_2$ is non-increasing on a subinterval of $I$, and then non-decreasing in the rest of $I$. 

Thus one can apply Corollary 4.3 to $u_1-u_2$ and conclude $\Delta u_2$ is constant on $I$.  A symmetric argument shows that $\Delta u_2$ is constant on $I^c$. This is only possible if $\partial\{u_2>u_3\}$ coincides with $\partial\{u_1>u_{2}\}$.

It is obvious how to inductively repeat this argument and show that $\partial\{u_k>u_{k+1}\}$ aligns with the same line for all $k$. Suppose the unit normal to this line is $e$, then a similar argument as in the proof of Corollary 4.3  gives $u_k-u_{k+1}$ only depends on $x\cdot e$. From there, homogeneity gives the desired result.  \end{proof} 

As a simple illustration of the result, we give a complete classification of homogeneous solution for the 3-membranes problem in the plane:

\begin{ex}
Let $\{u_j\}_{j=1,2,3}$ be solution to the 3-membranes problem in $\mathbb{R}^2$ with forcing terms $f_1=1$, $f_2=0$, and $f_3=-1$. Suppose they are homogeneous of degree 2 and satisfy (NA). 

If $0\in\partial\{u_1>u_2\}\cap\partial\{u_2>u_3\}$, then they fall into one of the following categories:

i) For some $e\in\mathbb{S}^1$, $u_1(x)=\frac{1}{2}(x\cdot e)_+^2$, $u_2(x)=0$, and $u_3(x)=-\frac{1}{2}(x\cdot e)_+^2$; 

ii) For some $e\in\mathbb{S}^1$, $u_1(x)=\frac{1}{2}(x\cdot e)_+^2+\frac{1}{4}(x\cdot e)_-^2$, $u_2(x)=-\frac{1}{4}(x\cdot e)_+^2+\frac{1}{4}(x\cdot e)_-^2$, and $u_3(x)=-\frac{1}{4}(x\cdot e)_+^2-\frac{1}{2}(x\cdot e)_-^2$;

iii) For some $e\in\mathbb{S}^1$ and $A\in S_{2,2}$ with $tr(A)=1$, $u_1(x)=\frac{1}{4}[(x\cdot e)_+^2+\langle Ax,x\rangle]$, $u_2(x)=\frac{1}{4}[-(x\cdot e)_+^2+\langle Ax,x\rangle]$, and $u_3(x)=-\frac{1}{2}\langle Ax,x\rangle$;

iv) For some $e\in\mathbb{S}^1$ and $A\in S_{2,2}$ with $tr(A)=1$, $u_1(x)=\frac{1}{2}\langle Ax,x\rangle$,  $u_2(x)=\frac{1}{4}[(x\cdot e)_+^2-\langle Ax,x\rangle]$ and $u_3(x)=\frac{1}{4}[-(x\cdot e)_+^2-\langle Ax,x\rangle]$;

v) For some $A_j\in S_{2,2}$ with $tr(A_1)=1$, $tr(A_2)=0$ and $tr(A_3)=-1$, $u_j(x)=\frac{1}{2}\langle A_jx,x\rangle$ for $j=1,2,3$. 

Moreover, solutions belong to these categories have different  Weiss-type energies in the order: $i)<ii)<iii)=iv)<v).$

\end{ex}

\section*{Acknowledgements}
H. Y. would like to thank Dennis Kriventsov for many fruitful discussions concerning this project.



\end{document}